\documentclass[11pt]{amsart} \usepackage{latexsym, amssymb, stmaryrd,amsthm}
\usepackage[T1]{fontenc}

\DeclareTextSymbol{\thh}{T1}{254}
\def\th{\textnormal{\thh}}

\newtheorem{thm}{Theorem}[section]
\newtheorem{lemma}[thm]{Lemma}
\newtheorem{prop}[thm]{Proposition}
\newtheorem{cor}[thm]{Corollary}

\theoremstyle{definition}

\newtheorem{rmk}[thm]{Remark}

\newtheorem{question}[thm]{Question}

\newcommand{\Z}{\mathbb{Z}}

\makeatletter

\def\indsym#1#2{%
  \setbox0=\hbox{$\m@th#1x$}%
  \kern\wd0%
  \hbox to 0pt{\hss$\m@th#1\mid$\hbox to 0pt{$\m@th#1^{#2}$}\hss}%
  \lower.9\ht0\hbox to 0pt{\hss$\m@th#1\smile$\hss}%
  \kern\wd0}

\def\nindsym#1#2{%
  \setbox0=\hbox{$\m@th#1x$}%
  \kern\wd0%
  \hbox to 0pt{\hss$\m@th#1\not$\kern1.4\wd0\hss}
  \hbox to 0pt{\hss$\m@th#1\mid$\hbox to 0pt{$\m@th#1^{\,#2}$}\hss}%
  \lower.9\ht0\hbox to 0pt{\hss$\m@th#1\smile$\hss}%
  \kern\wd0}

\def\dotminussym#1#2{%
  \setbox0=\hbox{$\m@th#1-$}%
  \kern.5\wd0%
  \hbox to 0pt{\hss\hbox{$\m@th#1-$}\hss}%
  \raise.6\ht0\hbox to 0pt{\hss$\m@th#1.$\hss}%
  \kern.5\wd0}
\newcommand{\dotminus}{\mathbin{\mathpalette\dotminussym{}}}

\def \r { {\mathbb R} }
\def \<{\langle}
\def \>{\rangle}

\def \*Z {{{^*}\Z}}

\def \((  {(\!(}
\def \)) {)\!)}

\def \acl{\operatorname{acl}}
\def \dcl{\operatorname{dcl}}

\def \int{\operatorname{int}}
\numberwithin{equation}{section}

\def \acl{\operatorname{acl}}

\def \u{\mathfrak{U}}
\def \UU{\mathbb{U}}

\def \iso{\operatorname{Iso}}
\def \id{\operatorname{id}}
\def \co{\colon}

\allowdisplaybreaks[2]

\begin{document}

\title{Definable Functions in Urysohn's Metric Space}

\author{Isaac Goldbring}

\address {University of California, Los Angeles, Department of Mathematics, 520 Portola Plaza, Box 951555, Los Angeles, CA 90095-1555, USA}
\email{isaac@math.ucla.edu}
\urladdr{www.math.ucla.edu/~isaac}

\begin{abstract}
Let $\u$ denote the Urysohn sphere and consider $\u$ as a metric structure in the empty continuous signature.  We prove that every definable function $\u^n\to\u$ is either a projection function or else has relatively compact range.  As a consequence, we prove that many functions natural to the study of the Urysohn sphere are not definable.  We end with further topological information on the range of the definable function in case it is compact.
\end{abstract}

\maketitle

\section{Introduction}

It is a common goal in model theory to understand the functions which are definable in a given structure, either by giving an explicit characterization of the definable functions in terms of non-logical verbiage or by proving structure theorems for definable functions.  An example of this phenomenon is the characterization of the functions definable in algebraically closed fields as exactly those functions whose graphs are constructible (i.e. Boolean combinations of Zariski closed sets).  Moreover, if the field is of characteristic $0$, then such a function is piecewise given by rational functions.

Model theory for metric structures is a recent generalization of classical model theory which is better suited for understanding structures based on complete metric spaces.  (For an introduction to the model theory of metric structures, see the wonderful survey \cite{BBHU}.)  As in classical model theory, there is an appropriate notion of definability of functions in metric structures.  However, there has yet to appear an analysis of the functions definable in a particular metric structure.  The goal of the present paper is to fill this void by studying the functions definable in the \emph{Urysohn sphere}.

The Urysohn sphere $\u$ is the unique (up to isometry) Polish metric space (i.e. complete, separable metric space) of diameter $1$ which is \emph{universal}, in the sense that every Polish metric space of diameter at most $1$ admits an isometric embedding into $\u$, and \emph{ultrahomogeneous}, in the sense that every isometry between finite subsets of $\u$ extends to an isometry of $\u$.  The Urysohn \emph{space}, which is the same as the Urysohn sphere without the bounded diameter requirement, as well as its isometry group, have been of interest to descriptive set theorists for a plethora of reasons; see, for example, \cite{Julien}.  

More recently, model theoretic aspects of the Urysohn sphere (as a metric structure) have been studied by Henson, Usvyatsov (see \cite{Usvy}) and Ealy and Goldbring (see \cite{Gold}).  The present paper intends to add to the understanding of the model theory of the Urysohn sphere by proving that definable functions $\u^n\to \u$ are either projection functions or else have a relatively compact range of a very special nature.  This structure theorem is the content of Section 2 below.  In Section 3, we study some consequences of our structure theorem.  In particular, we prove that many of the functions which appear naturally in the geometry of the Urysohn sphere are not definable.  We also show that there are no definable group operations on $\u$.  Finally, in Section 4, we prove some topological properties of the range of a definable function in case it is relatively compact.

I would like to thank Julien Melleray for many helpful discussions concerning the geometry of the Urysohn space.

\

\noindent \textbf{Notations and Background}

\
  
In the remainder of this introduction we establish some conventions as well as collect facts about definable functions that we will need for the main results.  

Fix a metric space $M$.  For $A\subseteq M$, we let $\bar{A}$ (resp. $\int(A)$) denote the closure (resp. interior) of $A$ in $M$.  For $a\in M$ and $r\in \r^{>0}$, we let $B_M(a;r)$ (resp. $\bar{B}_M(a;r)$) denote the open (resp. closed) ball in $M$ centered at $a$ of radius $r$.  We let $\id_M$ denote the identity function $M\to M$.

We let $\u$ denote the Urysohn sphere and consider it as a structure in the empty \emph{bounded} metric signature, that is the bounded metric signature consisting only of the metric symbol $d$, which is assumed to satisfy $d\leq 1$.  It is known that $\operatorname{Th}(\u)$ admits quantifier elimination (see \cite{Usvy}).  We let $\UU$ denote an $\omega_1$-universal domain for $\operatorname{Th}(\u)$ and we view $\u$ as an elementary substructure of $\UU$.  Note that $\u$ is closed in $\UU$ since $\u$ is complete.  Also note that $\omega_1$-saturation implies that any closed ball in $\UU$ (including $\UU$ itself) is not separable.  We equip $\UU^n$ with the maximum metric, that is, for $a=(a_1,\ldots,a_n),b=(b_1,\ldots,b_n)\in \UU^n$, we have
$$d(a,b)=\max_{1\leq i\leq n} d(a_i,b_i).$$ For all $i\in \{1,\ldots,n\}$, we let $\pi_i:\UU^n\to\UU$ denote the projection map onto the $i^{\text{th}}$ coordinate.

Recall that, for $A\subseteq \u$,  a function $f:\u^n\to\u$ is \emph{$A$-definable} if there is a countable sequence of formulae $(\varphi_k(x,y) \ | \ k<\omega)$ with parameters from $A$ such that the intepretations of the $\varphi_k$'s in $\u$, which are functions $\varphi_k^{\u}\co \u^{n+1}\to [0,1]$, converge uniformly to the function $d(f(x),y)$.  Since each $\varphi_k$ can only mention a finite number of elements of $A$, we can always take $A$ to be countable.  The sequence of functions $(\varphi_k^\UU)$ will also converge uniformly to a predicate, whose zeroset is the graph of a function, which we denote by $\tilde{f}:\UU^n\to\UU$ and call the \emph{natural extension of $f$ to $\UU$}.  By the construction of $\tilde{f}$, it is clear that $\tilde{f}$ is also an $A$-definable function.  Both $f$ and $\tilde{f}$ are uniformly continuous and we will let $\Delta:(0,1]\to (0,1]$ denote a modulus of uniform continuity for both $f$ and $\tilde{f}$.  (For proofs of all of these facts, see Section 9 of \cite{BBHU}.)  Finally, if $X$ is a definable subset of $\UU^n$, then, by $\omega_1$-saturation, $\tilde{f}(X)$ is a closed subset of $\UU$.  (In fact, $\tilde{f}(X)$ is a definable subset of $\UU$, although we will not need this fact; see \cite{B}, Lemma 1.20.)

As in classical model theory, if $f:\u^n\to\u$ is $A$-definable, then, for all $x\in \UU^n$, we have $\tilde{f}(x)\in \dcl(Ax)$.  However, Henson proved that $\dcl(B)=\acl(B)=\bar{B}$ for any $B\subseteq \UU$. (For a proof of this last fact, see Fact 5.3 of \cite{Gold}.)  Thus, for $x=(x_1,\ldots,x_n)\in \UU^n$, we have $\tilde{f}(x)\in \bar{A}\cup \{x_1,\ldots,x_n\}$.  This last fact is the key ingredient in the proof of our structure theorem.  



\section{Structure Theorem for Definable Functions}

Before we prove our main structure theorem for definable functions, we need a few preparatory lemmas.  Recall that a metric space $M$ is said to be \emph{finitely injective} if whenever $a_1\ldots,a_n\in M$ and $\{a_1,\ldots,a_n,a\}$ is an abstract one-point metric extension of $\{a_1,\ldots,a_n\}$, then there is $a'\in M$ such that $\{a_1,\ldots,a_n,a\}$ is isometric to $\{a_1,\ldots,a_n,a'\}$.

\begin{lemma}\label{L:pc}
Suppose that $(x_i \ | \ i<\omega)$ is a sequence from $\UU$ and $(r_i \ | \ i<\omega)$ is a sequence from $\r^{>0}$.  Let $B=\bigcup_{i<\omega} B_\UU(x_i;r_i)$.  Then $\UU\setminus B$ is finitely injective.
\end{lemma}

\begin{proof}
Fix $a_1,\ldots,a_n\in \UU\setminus B$ and let $\{a_1,\ldots,a_n,a\}$ be a one-point metric extension of $\{a_1,\ldots,a_n\}$.  By $\omega_1$-saturation, it suffices to prove that, for any $m<\omega$, the partial type 
$$\Gamma=\{d(x,a_i)=d(a,a_i) \ | \ 1\leq i\leq n\} \cup \{d(x, x_i)\geq r_i \ | \ 1\leq i\leq m\}$$  is finitely satisfiable in $\UU$.  Consider the one-point metric exension $$\{a_1,\ldots,a_n,x_1,\ldots,x_m,x\}$$ of $\{a_1,\ldots,a_n,x_1,\ldots,x_m\}$ given by:
\begin{itemize}
\item $d(x,a_i)=d(a,a_i)$ for each $i\in\{1,\ldots,n\}$, and 
\item $d(x,x_j)=\min_{1\leq k \leq n} (d(a,a_k)+d(a_k,x_j))$ for each $j\in \{1,\ldots,m\}$.
\end{itemize}
Then since $\UU$ is finitely injective, we can find $a'\in \UU$ realizing this extension.  Since $d(a',x_j)\geq \min_{1\leq k \leq n} d(a_k,x_j)\geq r_j$, it follows that $\Gamma$ is finitely satisfiable.  
\end{proof}

We will only need a consequence of the preceding lemma, namely that $\UU\setminus B$ is path-connected.  (To see how finite injectivity implies path-connectedness, see Sections 3 and 4 of \cite{Julien}.)  We will actually need path-connectedness of the complement of countably many balls in higher dimensions, which is the subject of our next lemma.  Let $\epsilon_0$ be a positive real number such that $\UU$ cannot be covered by finitely many balls of radius $\epsilon_0$.

\begin{lemma}
Suppose that $((x_1^i,\ldots,x_n^i) \ | \ i<\omega)$ is a sequence from $\UU^n$.  Suppose that $(r_i \ | \ i<\omega)$ is a sequence from $\r^{>0}$ such that $r_i<\epsilon_0$ for each $i<\omega$.  Let $B=\bigcup_{i<\omega} B_{\UU^n}((x_1^i,\ldots,x_n^i);r_i)$.  Then $\UU^n\setminus B$ is path-connected.  
\end{lemma}

\begin{proof}
Fix $a,b\in \UU^n\setminus B$.  Then for every $i<\omega$, there are $j,k\in\{1,\ldots,n\}$ such that $d(a_j,x_j^i),d(b_k,x_k^i)\geq r_i$.  For each $j\in \{1,\ldots,n\}$, let $c_j\in \UU\setminus \bigcup_{i<\omega} B_\UU(x_j^i;r_i)$; this is possible by saturation and the assumption on the $r_i$.  For each $j\in \{1,\ldots,n\}$, let $\Phi_j:[0,1]\to \UU$ be a continuous path connecting $a_j$ to $c_j$ such that, for all $i<\omega$, if $d(a_j,x_j^i)\geq r_i$, then $\Phi_j(s)\notin B_\UU(x_j^i; r_i)$ for all $s\in [0,1]$; this is possible by Lemma \ref{L:pc}.  Let $\Phi=(\Phi_1,\ldots,\Phi_n):[0,1]\to \UU^n$.  Then $\Phi$ is a continuous path connecting $(a_1,\ldots,a_n)$ to $(c_1,\ldots,c_n)$ which remains in $\UU^n\setminus B$.  One can connect $(c_1,\ldots,c_n)$ to $(b_1,\ldots,b_n)$ in a similar manner.  Concatenating these two paths yields the desired result.
\end{proof}

\begin{lemma}\label{L:sep}
Suppose that $F$ is a closed subset of $\UU^n$ and $G$ is a closed, separable subset of $F$ for which $F\setminus G\subseteq \int(F)$.  Then either $F=G$ or $F=\UU^n$.
\end{lemma}

\begin{proof}
Suppose $F\not=G$.  Let $y\in F\setminus G$.  Let $0<r<\min(d(y,G),\epsilon_0)$.  Cover $G$ with countably many balls of radius $r$ and call the union of these balls $B$.  Set $Y=\UU^n\setminus B$, which is path-connected by the previous lemma.  Now $F\cap Y=\int(F)\cap Y$ is a nonempty, clopen subset of $Y$, implying that $F\cap Y=Y$.  It follows that $Y\subseteq F$.  Since $r$ can be taken to be arbitrarily small, this shows that $\UU^n\setminus G\subseteq F$, whence $F=\UU^n$.
\end{proof}

For the rest of this section, $A$ denotes a countable subset of $\u$.  We are now ready to state the main theorem on definable functions.

\begin{thm}\label{T:many}
If $f:\u^n\to\u$ is an $A$-definable function, then either:
\begin{enumerate}
\item $\tilde{f}(\UU^n)$ is a compact subset of $\bar{A}$, or 
\item $\tilde{f}=\pi_i$ for some $i\in \{1,\ldots,n\}$.
\end{enumerate}
\end{thm}





Our first lemma shows that compactness of the image of $\tilde{f}$ follows from the image of $\tilde{f}$ being contained in $\bar{A}$.

\begin{lemma}
Suppose that $\tilde{f}(\UU^n)\subseteq \bar{A}$.  Then $\tilde{f}(\UU^n)$ is compact.
\end{lemma}

\begin{proof}
Since $\tilde{f}(\UU^n)$ is closed (and hence complete), it suffices to show that $\tilde{f}(\UU^n)$ is totally bounded.  Fix $\delta>0$.  Let $(a_i \ | \ i<\omega)$ enumerate $A$.  Let $\varphi(x,y)$ be a formula such that $|\varphi(x,y)-d(\tilde{f}(x),y)|\leq \frac{\delta}{4}$ for all $x\in \UU^n$ and $y\in \UU$.  Since for every $x\in \UU^n$ there is $i<\omega$ such that $d(\tilde{f}(x),a_i)<\frac{\delta}{4}$, the collection of closed conditions $$\{\varphi(x,a_i)\geq \frac{\delta}{2} \ | \ i<\omega\}$$ is not satisfied in $\UU$.  Thus, by $\omega_1$-saturation, there $a_1,\ldots,a_n\in A$ such that, for every $x\in \UU^n$, there is $i\in \{1,\ldots,n\}$ satisfying $\varphi(x,a_i)<\frac{\delta}{2}$.  It follows that $\tilde{f}(\UU^n)\subseteq \bigcup_{i=1}^n B_\u(a_i;\delta)$.
\end{proof}

We now prove Theorem \ref{T:many} in the case $n=1$.  Suppose that $f:\u\to\u$ is an $A$-definable function.  Set $X=\{x\in \u \ | f(x)=x\}$.  

\begin{lemma}\label{L:int}
$\tilde{f}^{-1}(\bar{A})\setminus (X\cap \bar{A})\subseteq \int(\tilde{f}^{-1}(\bar{A}))$.
\end{lemma}

\begin{proof}
Suppose that $x\in \UU$ is such that $\tilde{f}(x)\in \bar{A}$ and $\tilde{f}(x)\not=x$.  Let $r=d(\tilde{f}(x),x)>0$.  Let $\delta=\min\{\frac{r}{2},\Delta(\frac{r}{2})\}$.  Suppose $y\in \UU$ is such that $d(x,y)<\delta$.  Then $d(\tilde{f}(x),\tilde{f}(y))\leq \frac{r}{2}$.  Suppose that $\tilde{f}(y)=y$.  Then $$d(x,\tilde{f}(x))\leq d(x,y)+d(y,\tilde{f}(x))<\frac{r}{2}+\frac{r}{2}=r,$$ a contradiction.  Hence $\tilde{f}(y)\in \bar{A}$, finishing the proof of the lemma.
\end{proof}

By Lemmas \ref{L:sep} and \ref{L:int}, it follows that either $\tilde{f}^{-1}(\bar{A})=X\cap \bar{A}$, in which case $\tilde{f}=\id_\UU$, or $\tilde{f}^{-1}(\bar{A})=\UU$.  Hence, Theorem \ref{T:many} holds for one-variable functions.  

We now prove Theorem \ref{T:many} by induction on $n$.  Suppose that $f:\u^n\to\u$ is an $A$-definable function, where $n>1$, and suppose that the conclusion of Theorem \ref{T:many} holds for definable functions of $n-1$ variables.  

First observe that $\UU\setminus \u$ is path-connected.  Indeed, given $a,b\in \UU\setminus \u$, we can find $\epsilon>0$ such that $d(a,\u),d(b,\u)\geq \epsilon$.  Let $B$ be the union of countably many open balls of radius $\epsilon$ which cover $\u$.  Then by Lemma \ref{L:pc}, we can connect $a$ and $b$ by a path which remains in $\UU\setminus B$ and hence in $\UU\setminus \u$.  It follows that $(\UU\setminus \u)^n$ is path-conncted.  Let us also observe that $\UU\setminus \u$ is dense in $\UU$.  Indeed, let $x\in \u$ and $\epsilon>0$.  Since $B_\UU(x;\epsilon)$ is not separable while $B_\u(x;\epsilon)$ is separable, there is $y\in \UU\setminus \u$ with $d(x,y)<\epsilon$.  It then follows that $(\UU\setminus \u)^n$ is dense in $\UU^n$.  

Now suppose that $x=(x_1,\ldots,x_n)\in (\UU\setminus \u)^n$ and $\tilde{f}(x)\in \bar{A}$.  Let $\epsilon>0$ be such that $B_\UU(x_i;\epsilon)\subseteq \UU\setminus \u$ for each $i\in \{1,\ldots,n\}$.  Let $\delta=\min(\frac{\epsilon}{2},\Delta(\frac{\epsilon}{2}))$.  Then if $y=(y_1,\ldots,y_n)\in \UU^n$ is such that $d(x,y)<\delta$, then $\tilde{f}(y)\in \bar{A}$. Consequently, $\tilde{f}^{-1}(\bar{A})\cap (\UU\setminus \u)^n\subseteq \int(\tilde{f}^{-1}(A))$.  It follows that $$\tilde{f}^{-1}(\bar{A})\cap (\UU\setminus \u)^n=\int(\tilde{f}^{-1}(\bar{A}))\cap (\UU\setminus \u)^n$$ is a clopen subset of $(\UU\setminus \u)^n$, whence $\tilde{f}^{-1}(\bar{A})\cap (\UU\setminus \u)^n=\emptyset$ or $(\UU\setminus \u)^n$.

Suppose first that $\widetilde{f}^{-1}(\bar{A})\cap (\UU\setminus \u)^n=(\UU\setminus \u)^n$.  Then since $(\UU\setminus \u)^n$ is dense in $\UU^n$, we have that $\widetilde{f}(\UU^n)\subseteq \bar{A}$.  We may thus suppose that $\widetilde{f}^{-1}(\bar{A})\cap (\UU\setminus \u)^n=\emptyset$.  It follows that, for every $x=(x_1,\ldots,x_n)\in (\UU\setminus \u)^n$, there is $i\in \{1,\ldots,n\}$ such that $\tilde{f}(x)=x_i$.  Since $(\UU\setminus \u)^n$ is dense in $\UU^n$, it follows that:
\begin{equation}
\text{for each }x\in \UU^n\text{, there is }i\in \{1,\ldots,n\} \text{ such that }\tilde{f}(x)=x_i. \tag{$\dagger$}
\end{equation}

For $a\in \u^{n-1}$, consider the $Aa$-definable function $f_a:\u\to\u$ given by $f_a(b)=f(a,b)$.  Note that the natural extension $\widetilde{f_a}:\UU\to\UU$ satisfies $\widetilde{f_a}(b)=\tilde{f}(a,b)$ for all $b\in \UU$.  Since we know Theorem \ref{T:many} holds for one-variable functions, we have that either $\widetilde{f_a}=\id_\UU$ or $\widetilde{f_a}(\UU)\subseteq \overline{Aa}$.

\begin{lemma}\label{L:conn}
Let $Y=\{a\in \u^{n-1} \ | \ \widetilde{f_a}=\id_\UU\}$.  Then $Y=\emptyset$ or $\u$.
\end{lemma}  

\begin{proof}
Clearly $Y$ is a closed subset of $\u^{n-1}$.  We now show that $Y$ is also an open subset of $\u^{n-1}$, from which the lemma follows using the connectedness of $\u^{n-1}$.  Take $\epsilon>0$ such that the closed $\epsilon$-neighborhood of $\u$ in $\UU$ does not equal $\UU$; such an $\epsilon$ exists since $\UU$ is not separable.  Let $\delta=\Delta(\epsilon)$.  Suppose $a\in Y$ and $b\in \u^{n-1}$ is such that $d(a,b)<\delta$.  Then $$d(x,\widetilde{f_b}(x))=d(\tilde{f}(a,x),\tilde{f}(b,x))\leq \epsilon$$ for all $x\in \UU$.  This prevents $\widetilde{f_b}(\UU)\subseteq \overline{Ab}\subseteq \u$, whence $\widetilde{f_b}=\id_\UU$.     
\end{proof}

By the previous lemma, if there is $a\in \u^{n-1}$ such that $\widetilde{f_a}=\id_\UU$, then $f=\pi_n\upharpoonright \u^n$, whence $\tilde{f}=\pi_n$ and the proof of Theorem \ref{T:many} would be complete.  We may thus suppose that $\widetilde{f_a}(\UU)\subseteq \overline{Aa}$ for all $a\in \u^{n-1}$.  Now fix $a=(a_1,\ldots,a_{n-1})\in \u^{n-1}$ and $b\in \UU\setminus \u$.  We know that $\widetilde{f_a}(b)\in \overline{Aa}$ and, by ($\dagger$), $\tilde{f}(a,b)\in \{a_1,\ldots,a_{n-1},b\}$.  It follows that $\widetilde{f_a}(b)\in \{a_1,\ldots,a_{n-1}\}$.  Since $\UU\setminus \u$ is dense in $\UU$, it follows that $\widetilde{f_a}(\UU)\subseteq \{a_1,\ldots,a_{n-1}\}$.  Since $a\in \u^{n-1}$ is arbitrary, it follows that:
\begin{equation}
\text{ for each }a=(a_1,\ldots,a_{n-1})\in \u^{n-1}\text{, we have }\widetilde{f_a}(\UU)\subseteq \{a_1,\ldots,a_{n-1}\}. \tag{$\dagger \dagger$}
\end{equation}Now fix $b\in \u$ and consider the $Ab$-definable function $f^b:\u^{n-1}\to\u$ given by $f^b(a)=f(a,b)$.  By induction, either $f^b(\u^{n-1})$ is a relatively compact subset of $\overline{Ab}$ or else there is $i\in \{1,\ldots,n-1\}$ such that $f^b(a)=a_i$ for all $a=(a_1,\ldots,a_{n-1})\in \u^{n-1}$.  We cannot have the former option.  Indeed, if $f^b(\u^{n-1})$ were relatively compact, then there would be $c\in \u\setminus f^b(\u^{n-1})$.  But then $f(c,c,\ldots,c,b)=c$ by $(\dagger \dagger)$, a contradiction.  Thus there is (a unique) $i\in \{1,\ldots,n-1\}$ such that $f^b(a)=a_i$ for all $a=(a_1,\ldots,a_{n-1})\in \u^{n-1}$.  For $i=1,\ldots,n-1$, set $Y_i:=\{b\in \u \ | \ f^b(a)=a_i \text{ for all }a\in \u^{n-1}\}$.  We have thus shown that $\u$ is the disjoint union of the closed sets $Y_1,\ldots,Y_{n-1}$, whence by the connectedness of $\u$, there is $i\in \{1,\ldots,n-1\}$ such that $f(x)=x_i$ for all $x=(x_1,\ldots,x_n)\in \u^n.$  This finishes the proof of Theorem \ref{T:many} \qed

\begin{rmk}
Since compact sets are the analogue of finite sets in continuous logic, the first option in the conclusion of Theorem \ref{T:many} is almost like saying that $\tilde{f}$ is piecewise constant.  However, one must remember $\tilde{f}(\UU)$ is also connected, so if $\tilde{f}(\UU)$ is actually finite, then $\tilde{f}$ is a constant function.
\end{rmk}

\section{Consequences of the Main Result}

The main result on definable functions tells us that many functions are not definable.  For example:

\begin{cor}
If $f:\u\to\u$ is a definable surjective/open/proper map, then $f=\id_\u$.
\end{cor}

\begin{proof}
The statements about definable surjective and proper maps are clear.  To see why a definable open map must necessarily be the identity, it remains to observe that relatively compact subsets of $\u$ have empty interior.
\end{proof}

\begin{cor}
If $f:\u\to\u$ is a definable isometric embedding, then $f=\id_\u$.
\end{cor}

We now expound on the significance of the previous corollary.  There are many interesting isometric embeddings $\u\to\u$.  For example, any isometry $\phi$ between compact subsets $C$ and $C'$ of $\u$ can be extended to an isometry $\Phi$ of $\u$.  The previous corollary shows that if $\phi\not=\id_C$, then $\Phi$ is not definable.  

Due to the homogeneity of $\u$, there are many proper subsets of $\u$ isometric to $\u$.  For example, given any $x_1,\ldots,x_n\in \u$, the set
$$\operatorname{Med}(x_1,\ldots,x_n)=\{z\in \u \ | \ d(z,x_i)=d(z,x_j) \text{ for all }i,j=1,\ldots,n\}$$ is isometric to $\u$.  Also, if $M$ is any Polish subspace of $\u$ which is a \emph{Heine-Borel} metric space (i.e. closed balls in $M$ are compact) and $R\in (0,1]$, then $$\{x\in \u \ | \ d(x,M)\geq R\}$$ is isometric to $\u$.  (See \cite{Julien} for proofs of all of the results mentioned in this paragraph.)  The previous corollary shows that if $Z$ is a proper subspace of $\u$ which is isometric to $\u$, then any isometry $\Phi:\u\to Z$ is not definable.  

On a related note, by the universality of $\u$, for each $n\geq 1$, there is an isometric embedding $\u^n\to\u$.  It was proven in Corollary 5.16(1) of \cite{Gold} that, for all $n\geq 2$, there are no \emph{definable} isometric embeddings $\u^n\to\u$.

Here is one more corollary of Theorem \ref{T:many} with a topological flavor.

\begin{cor}
If $f:\u\to\u$ is a definable contraction mapping, then $f(\u)$ is a relatively compact subset of $\bar{A}$.
\end{cor}

The structure theorem for definable functions also gives us information on definable groups.

\begin{cor}
There are no definable group operations on $\u$.
\end{cor}

\begin{proof}
Indeed, a group operation is a function of two variables which is surjective but not equal to a projection function.
\end{proof}

In particular, the group operation on $\u$ introduced by Vershik and Cameron which makes $\u$ into a monothetic Polish group (see \cite{CV}) is not definable.

Our next series of corollaries involve functions on $\u$ induced by functions on $\ell^2$.  In \cite{Usp}, Uspenskij proved that $\u$ is homeomorphic to $\ell^2$.  Let $\Phi:\ell^2\to\u$ be a homeomorphism.  Suppose $F:\ell^2\to\ell^2$ is a function.  Define $f:\u\to\u$ by $f(x)=\Phi(F(\Phi^{-1}(x)))$ and call $f$ the function \emph{induced by} $F$.

\begin{cor}
Suppose that $F\co\ell^2\to\ell^2$ is a continuous, linear surjection.  Suppose that the induced map is definable.  Then $F=\id_{\ell^2}$.
\end{cor}

\begin{proof}
By the Open Mapping Theorem, $F$, and consequently the induced map, would be an open map.
\end{proof}

In particular, the function on $\u$ induced by multiplication by a fixed scalar is not definable.  The next corollary is proven in a similar way.

\begin{cor}
If $x\in \ell^2$, then translation by $x$ does not induce a definable function on $\u$.
\end{cor}

Recently, Melleray \cite{Julien2} proved that $\operatorname{Iso}(\u)$ is homeomorphic to $\ell^2$, and hence, combined with Uspenskij's result refererred to above, $\iso(\u)$ is homeomorphic to $\u$.  Let $\Psi:\iso(\u)\to\u$ be a homeomorphism.  We can thus view the evaluation map $(\phi,x)\mapsto \phi(x)\co \iso(\u)\times \u \to \u$ as a map $(\Psi(\phi),x)\mapsto \phi(x) \co \u^2\to \u$.  Let us call the latter map the \emph{induced evaluation map}.

\begin{cor}
The induced evaluation map is never definable.
\end{cor}

\begin{proof}
Since the induced evaluation map is surjective, we have that the induced evaluation map is either $(\Psi(\phi),x)\mapsto \Psi(\phi)$ or $(\Psi(\phi),x)\mapsto x$.  The former option implies that $\Psi(\phi)=\phi(x)$ for all $x\in \u$, contradicting the injectivity of $\phi$, while the latter option implies that $\phi=\id_\u$ for all $\phi\in \iso(\u)$, again a contradiction.
\end{proof}

We end this section with two results concerning the case that all elements of the defining parameterset $A$ of $f$ are fixed by $f$.

\begin{cor}
Suppose that $f\co\u\to\u$ is $A$-definable, $f(a)=a$ for all $a\in A$, and $\bar{A}$ is not compact.  Then $f=\id_\u$.
\end{cor}
 
\begin{cor}
Suppose that $f\co\u\to\u$ is $A$-definable, $f(a)=a$ for all $a\in A$, and there is some separable metric space $B$ containing $\bar{A}$ as a closed subspace such that $\bar{A}$ is not a retract of $B$.  Then $f=\id_\u$. 
\end{cor}

\begin{proof}
By universality of $\u$, we can realize $B$ inside of $\u$.  If $f\not=\id_\u$, then $f(\u)=\bar{A}$, whence $\bar{A}$ is a retract of $\u$, and hence of $B$.
\end{proof}

As a specific instance of the previous corollary, we see that if $f$ is $A$-definable, where $\bar{A}=S^n$, the $n$-dimensional sphere, and $f(a)=a$ for all $a\in A$, then $f=\id_\u$.  Indeed, $S^n$ is not a retract of the closed $(n+1)$-dimensional unit ball $B^{n+1}$.

\section{The Case of Compact Range}

In this section, we assume that $f:\u^n\to \u$ is an $A$-definable function and $\tilde{f}(\UU^n)$ is a compact subset of $\bar{A}$.  The goal of this section is to obtain further topological information about $\tilde{f}(\UU^n)$.  Since compact subsets of $\UU$ have empty interior, we immediately have the following:

\begin{lemma}
$\tilde{f}(\UU^n)$ has no interior in $\UU$.
\end{lemma}

Since $\UU^n$ is path-connected and the continuous image of a path-connected space is path-connected, we get

\begin{lemma}
$\tilde{f}(\UU^n)$ is path-connected.
\end{lemma}

In particular, $\tilde{f}(\UU^n)$ is connected, whence it is a \emph{continuum}.  

\begin{cor}
If $\bar{A}$ is totally disconnected, then $\tilde{f}$ is a constant function.  
\end{cor}

The connectedness of $\tilde{f}(\UU^n)$ also immediately implies the following:

\begin{cor}
If $\tilde{f}$ is not a constant function, then $\tilde{f}(\UU^n)$ is perfect.
\end{cor}

Thus, we see that $|\tilde{f}(\UU^n)|=1$ or $2^{\aleph_0}$.  We now investigate further the properties that $\tilde{f}(\UU^n)$ has as a continuum.



Recall the following two definitions from the theory of continua.  First, a \emph{Peano space} is a continuous image of the unit interval in a hausdorff space.  Second, a continuum $C$ is said to be \emph{irreducible} if there are distinct $x,y\in C$ such that no proper subcontinuum of $C$ contains both $x$ and $y$.  If a continuum is not irreducible, then it is said to be \emph{reducible}.  The next lemma follows immediately from the path-connectedness of $\tilde{f}(\UU^n)$.

\begin{lemma}
Either $\tilde{f}(\UU)$ is a Peano space or a reducible continuum. 
\end{lemma}

It follows from the previous lemma and Corollaries 3-42 and 3-49 of \cite{HY} that if $\tilde{f}(\UU^n)$ is not a singleton, then it is \emph{decomposable} as a continuum, meaning that it is the union of two proper subcontinua.  Since the generic continuum is indecomposable (as the set of continuum homeomorphic to a certain indecomposable continuum, namely the \emph{pseudo-arc}, is a dense $G_\delta$ subspace of the space of continua), we see that $\tilde{f}(\UU^n)$ is a special kind of continuum.

If $P$ is a property, we say that a continuum $C$ has \emph{arbitrarily small subcontinua with property $P$} if for any $y\in C$ and any $\epsilon>0$, $B_C(y;\epsilon)$ contains a subcontinuum of $C$ with property $P$ other than $\{y\}$.

\begin{lemma}
If $\tilde{f}(\UU^n)$ is not a singleton, then $\tilde{f}(\UU^n)$ has arbirarily small path-connected subcontinua.
\end{lemma}

\begin{proof}
Suppose that $C:=\tilde{f}(\UU^n)$ doesn't have arbitrarily small path-connected subcontinua.  Then there is $y\in C$ and $\epsilon>0$ such that the only path-connected subcontinuum of $C$ contained in $B_C(y;\epsilon)$ is $\{y\}$.  Now suppose that $x=(x_1,\ldots,x_n)\in \tilde{f}^{-1}(y)$ and $\delta<\Delta(\frac{\epsilon}{2})$.  Then $\tilde{f}(\bar{B}_{\UU^n}(x;\delta))$ is a path-connected subcontinuum of $C$ contained in $B_C(y;\epsilon)$.  Indeed, $\bar{B}_{\UU^n}(x;\delta)$ is a path-connected subset of $\UU$, whence its image under $\tilde{f}$ is path-connected.  Also, as $\UU$ is a geodesic space, $\bar{B}_\UU(x_i;\delta)$ is a definable subset of $\UU$ for each $i\in \{1,\ldots,n\}$ (see Lemma 1.8.15 of \cite{Car}).  Thus, $\bar{B}_{\UU^n}(x;\delta)$ is a cartesian product of definable sets, whence definable by  Lemma 1.10 of \cite{B}.  Thus $\tilde{f}(\bar{B}_{\UU^n}(x;\delta))$ is a closed subset of $C$.  It follows that $\tilde{f}(\bar{B}_{\UU^n}(x;\delta))=\{y\}$ and hence $\tilde{f}^{-1}(y)$ is a clopen subset of $\UU$.  Thus $\tilde{f}(\UU^n)=\{y\}$. 
\end{proof}

It follows that if $\tilde{f}(\UU^n)$ is not a singleton, then $\tilde{f}(\UU^n)$ has arbitrarily small reducible subcontinua.

We end this section with two questions.

\begin{question}\label{L:question1}
Is it possible to improve our structure theorem to show that if $f:\u^n\to\u$ is a definable function, then either $f$ is a projection function or else $f$ is constant?  Said another way, is it impossible to find a non-degenerate continuum $C$ with the properties mentioned above such that there is a definable function $f:\u^n\to\u$ with $\tilde{f}(\UU^n)=C$?
\end{question}

\begin{question}\label{L:question2}
If $f:\u\to\u$ is an injective definable map, must $f=\id_\u$?
\end{question}

A few remarks are in order concerning the latter question.  First, one must note that Theorem \ref{T:many} does not yield an answer to Question \ref{L:question2}.  Indeed, it is possible to have a continuous, injective map $\u\to \u$ with relatively compact image.  Since $\u$ is homeomorphic to $\ell^2$, it suffices to construct a continuous, injective map $\ell^2\to\ell^2$ with relatively compact image.  Here is an example of such a function, as was communicated to me by Jan van Mill.  By Theorem 6.6.11 of \cite{Mill}, $\ell^2$ is homeomorphic to $\r^{\infty}$, which is in turn homeomorphic to $(0,1)^\infty$.  This latter space is homeomorphic to a subspace of the \emph{compact} Hilbert cube $[0,1]^\infty$, which itself is homeomorphic to a subspace of $\ell^2$ via the map $(x_n)\mapsto (\frac{x_n}{2^n})$.  The composition of these maps yields the desired example.  

However, Theorem \ref{T:many} yields a positive answer to Question \ref{L:question2} in the case that the parameterset $A$ defining $f$ is totally disconnected.  In this case, the answer to Question \ref{L:question1} is affirmative and an affirmative answer to Question \ref{L:question1} yields an affirmative answer to Question \ref{L:question2}.  It follows from this that there are no $A$-definable injective maps $f:\u^n\to\u$ for $n\geq 2$ and $A$ totally disconnected.  The aforementioned result generalizes Corollary 5.16(2) of \cite{Gold}, which has the stronger assumption that $A$ is finite.

In connection with Question \ref{L:question2}, one can prove the following:

\begin{prop}\label{T:inj}
If $\tilde{f}:\UU\to\UU$ is injective, then $\tilde{f}=\id_\UU$.
\end{prop}

\begin{proof}
One first observes that if $C$ is the complement of an open ball $B_\UU(e;r)$ in $\UU$, then $C$ is definable over $\{e\}$.  This fact has already been observed by Carlisle and Point, but since it has yet to appear in the literature, we provide a proof here.  Set $\phi(x,y):=d(x,y)\dotplus (r\dotminus d(y,e)).$  (Here, $a\dotplus b=\min(a+b,1)$ and $a\dotminus b=\max(a-b,0)$.)  Observe that $\phi(x,y)=d(x,y)$ for $y\in C$.  We next claim that $\phi(x,y)\geq d(x,C)$ for $y\notin C$.  If we are successful in proving this, then it follows that $d(x,C)=\inf_y\phi(x,y)$, proving that $C$ is $\{e\}$-definable.  Suppose $r_0:=d(e,y)<r$.  Let $\Phi:[0,r]\to \UU$ be an isometry such that $\Phi(0)=e$ and $\Phi(r_0)=y$.  Set $y':=\Phi(r)\in C$.  Then $d(x,C)\leq d(x,y')\leq d(x,y)+d(y,y')=d(x,y)+(r-d(y,e))=\phi(x,y)$.  (One should note that the definability of $C$ also follows from a very recent result of Melleray which shows that a closed subset $C$ of $\UU$ is definable over a countable $A\subseteq \UU$ if and only if $C$ is invariant under all automorphisms of $\UU$ which fix $A$.)  

Again suppose that $C$ is the complement of an open ball in $\UU$.  Then the preceding paragraph implies that $\tilde{f}(C)$ is a closed subset of $\UU$.  Since $\tilde{f}$ is injective, this implies that $\tilde{f}$ is a closed map, whence it is a topological embedding.  It follows from Theorem \ref{T:many} that $\tilde{f}=\id_\UU$.
\end{proof}

A quicker, but less elementary, way of proving the preceding proposition would be to use the machinery of $U^\th$-rank for continuous rosy theories as developed in \cite{Gold}.  Indeed, if $\tilde{f}$ is injective, then $U^{\th}_{\operatorname{real}}(\UU)=U^{\th}_{\operatorname{real}}(\tilde{f}(\UU)).$  It is shown in \cite{Gold} that $U^{\th}_{\operatorname{real}}(\UU)=1$.  This prevents $\tilde{f}(\UU)$ from being compact, for then $U^{\th}_{\operatorname{real}}(\tilde{f}(\UU))=0$.

Unfortunately, Proposition \ref{T:inj} does not settle the answer to Question \ref{L:question2}.  Indeed, since continuous logic is a positive logic, one cannot show that $\tilde{f}$ is injective given that $f$ is injective.  (This issue is discussed in Section 2 of \cite{Gold}.)

\end{document}